\documentclass{amsart}
\usepackage[latin1]{inputenc}
\usepackage[english]{babel}
\usepackage{csquotes}

\usepackage{color} 
\usepackage{hyperref}
\hypersetup{colorlinks=true,linktoc=all,linkcolor=blue,citecolor=red,filecolor=pink,urlcolor=blue}
\usepackage[hyperref=true,style=alphabetic,citestyle=alphabetic,backend=bibtex,maxnames=100,doi=false,isbn=false,url=false,giveninits=true]{biblatex}
\bibliography{biblio}

\usepackage{enumitem}
\usepackage{graphicx}
\graphicspath{{pictures/}}

\usepackage[mathcal,mathscr]{euscript}

\usepackage{tikz}
\usetikzlibrary{arrows}

\usepackage{amssymb,amsmath,latexsym}
\theoremstyle{plain}                    
\newtheorem{theorem}{Theorem}[section]
\newtheorem{lemma}[theorem]{Lemma}
\newtheorem{proposition}[theorem]{Proposition}

\theoremstyle{definition}

\newtheorem{example}[theorem]{Example}
\newtheorem{question}{Question}
\theoremstyle{remark}
\newtheorem{remark}[theorem]{Remark}

\numberwithin{equation}{section}
 
\newcommand{\zz}{\mathbb Z}

\newcommand{\rr}{\mathbb R}

\newcommand{\rp}[1]{\mathbb{RP}^{#1}}
\newcommand{\cp}[1]{\mathbb{CP}^{#1}}
\newcommand{\pglr}{\mathrm{PGL}(n+1,\rr)}
\newcommand{\conf}{\text{M\"ob}(n)} %this is PO(n,1)

\newcommand{\dev}{\mathrm{dev}}

\newcommand{\confsphere}{\mathbb S}

\newcommand{\wu}{\operatorname{SU}(3)/\operatorname{SO}(3)}
\newcommand{\hyperbolize}[1]{\mathcal H(#1)}
\newcommand{\cd}{h} %Charney-Davis map on the complex

\newcommand{\cat}[1]{\operatorname{CAT}(#1)}

\newcommand{\lk}[1]{\operatorname{lk}\left(#1\right)}

\makeatletter
\@namedef{subjclassname@2020}{\textup{2020} Mathematics Subject Classification}
\makeatother

\begin{document}
\title[Manifolds without real projective or flat conformal structures]{Manifolds without real projective \\ or flat conformal structures}

\author{Lorenzo Ruffoni}
\address{Department of Mathematics - Tufts University, 117 College Avenue, Medford, MA 02155, USA}
\email{lorenzo.ruffoni2@gmail.com}

\date{\today}

\subjclass[2020]{57N16, 20F67, 53C23}
 \keywords{Real projective structure; flat conformal structure; spherical CR structure; strict hyperbolization; relative hyperbolization; Gromov hyperbolic group; characteristic classes.}

\begin{abstract}
In any dimension at least five we construct examples of closed smooth manifolds with the following properties: 1) they have neither real projective nor flat conformal structures; 2) their fundamental group is a non-elementary Gromov hyperbolic group. These examples are obtained via relative strict hyperbolization.
\end{abstract}

\maketitle
% \tableofcontents
%%%%%%% prints list of todos	
% \makeatletter
% \providecommand\@dotsep{5}
% \makeatother
% \listoftodos\relax

\section{Introduction}
\addtocontents{toc}{\protect\setcounter{tocdepth}{2}}

The uniformization problem for smooth manifolds consists roughly speaking in the following question: given a smooth manifold $M$, can we endow $M$ with a locally homogeneous geometric structure?
Here, by a \textit{locally homogeneous geometric structure} we mean a structure that locally looks like the geometry of a homogeneous space $X$ for some Lie group $G$ (see \S\ref{sec:gxstructures} and \cite{GO10} for details).
For example, a Riemannian metric of constant sectional curvature provides a geometric structure locally modelled on the hyperbolic space, the Euclidean space, or the round sphere, depending on the sign of the curvature.
It is classical that every surface admits a metric of constant curvature, depending on the sign of its Euler characteristic.
But in higher dimension the topology of a manifold can obstruct the existence of constant curvature metrics. 
So, it is natural to look for generalizations, and locally homogeneous geometric structures provide a natural framework for this quest.

In this paper we are mainly interested in two types of geometric structures.
A \textit{real projective structure} is a locally homogeneous geometric structure modelled on the real projective space $\rp n$ and its group of real projective transformations $\pglr$. 
These are also known as \textit{flat projective structures} in the literature. 
A \textit{flat conformal structure} is a locally homogeneous geometric structure modelled on the sphere $\confsphere^ n$ and its group of M\"obius transformations $\conf$ (also known as conformal  transformations).
Notice that in both cases the geometry is not defined by a  Riemannian metric, but rather by the action of a Lie group, i.e. these should be thought as geometric structures in the sense of Ehresmann and Thurston.
Given the more general framework, one could ask if every manifold can be equipped with a real projective or flat conformal  structure.
In the positive direction, note that the constant curvature geometries admit models inside the real projective geometry of $\rp n$, and also inside the flat conformal geometry of the sphere $\confsphere ^n$. 
In particular, any surface has such a structure.
Moreover, many $3$-manifolds do as well (see \cite{SC83,MO97,MA98,KA89,BDL18}).

However, there are 3-manifolds that do not support this kind of geometric structures.
Goldman in \cite{GO83} provided the fist examples of $3$-manifolds without any flat conformal structure, and then Cooper and Goldman in \cite{CG15} gave the first example of a $3$-manifold without real projective structures.
The examples in \cite{GO83} are obtained as circle bundles over a torus, or as torus bundles over the circle. They admit $Nil^3$ or $Sol^3$ geometry, hence they admit real projective structures (see \cite{MO97}).
The example in \cite{CG15} is obtained as the connected sum $\rp 3 \# \rp 3$, and it follows from \cite{KU78} that it admits a flat conformal structure.

In higher dimension there is a simple way to produce examples of manifolds without real projective or flat conformal structures.
Indeed, a closed simply connected manifold admitting such a structure must be a sphere (because the developing map must be a covering map). 
But for $n\geq 4$ there are closed simply connected $n$--manifolds which are not spheres (e.g. $\cp n$ for $n\geq 2$).
So, the problem is to construct examples with infinite fundamental group.
For the real projective case, \c Coban (see  \cite{C21}) has generalized to any dimension $n\geq 4$  the work of Cooper and Goldman in \cite{CG15}. All these examples are closed smooth manifolds with fundamental group the infinite dihedral group $\zz_2  \ast \zz_2$.
\c Coban also provided examples with fundamental group $\zz$ in \cite{C19}.
To our knowledge, these are all the known examples of manifolds with no real projective structure.
Their fundamental groups are small, in the sense that they are virtually cyclic.
Similarly, we note that also the fundamental groups of the manifolds with no flat conformal structures provided by Goldman in \cite{GO83} and Maier in \cite{MA98} are small, in the sense that they are solvable. 
On the other hand, Kapovich in \cite{KA04} has shown how to obtain $4$-manifolds with no flat conformal structure and arbitrary fundamental group. 

\bigskip
The purpose of this paper is to construct new examples of higher dimensional manifolds that have neither  real projective nor flat conformal structures and have ``large'' fundamental group. 
More precisely, we prove the following.

\begin{theorem}\label{maintheorem}
For all $n\geq 5 $ there exist a closed smooth $n$--manifold $N$ such that
\begin{enumerate}
    \item $N$ has no real projective structure.
    \item $N$ has no flat conformal structure.
    \item $\pi_1(N)$ is a non-elementary Gromov hyperbolic group.

\end{enumerate}
\end{theorem}

\noindent The fundamental group of each of the manifolds in Theorem~\ref{maintheorem} is ``large'' in the sense that it has the following features.
\begin{enumerate}
\item Being a non--elementary Gromov hyperbolic group, it is not virtually solvable, and it is SQ-universal (i.e. every countable group embeds into one of its quotients, see \cite{OL95}). 
This in sharp contrast with the examples from  \cite{CG15,C21}), where a key ingredient in the construction is the fact that all the proper quotients of the fundamental group are finite.

\item It retracts to many free subgroups and does not have property (T). This is a by-product of the methods used to construct it (see \S\ref{sec:main} and \cite{BE07}).
\end{enumerate}

We conclude this introduction with some observations and  some questions.
The manifolds in Theorem~\ref{maintheorem} are constructed using the relative version of the strict hyperbolization procedure of Charney and Davis (see \cite{CD95,BE07}), but they are not just the hyperbolization of the manifolds constructed in \cite{CG15,C19,C21}, see Remark~\ref{rem:not strict but rel}.
Indeed, the manifolds in Theorem~\ref{maintheorem} are obtained starting from a triangulation of $M\times [0,1]$, where $M$ is a closed simply connected manifold of dimension $m=n-1$ that is not a sphere (hence the restriction to $n\geq 4$, see Remark~\ref{rem:dimension}).
In particular, since such an $M$ cannot be aspherical, our manifolds are themselves not aspherical.

\begin{question} 
Are there aspherical smooth manifolds with Gromov hyperbolic fundamental group, which admit neither real projective nor flat conformal structures?
\end{question}

We emphasize that our methods actually show that the manifolds in Theorem~\ref{maintheorem} do not admit any locally homogeneous geometric structure modelled on a sphere, or any manifold covered by a sphere. 
In odd dimension this includes for instance spherical CR-structures (see Remark~\ref{rem:obstruction all spheres}). 
From this perspective, the manifolds in Theorem~\ref{maintheorem} are not uniformizable in a very strong sense.
On the other hand, as observed above, the $3$-dimensional examples in \cite{GO83,CG15} admit at least one type of geometric structure.
Finally, the restriction to dimension $n\geq 5$ in Theorem~\ref{maintheorem} is due to the fact that a closed simply connected manifold of dimension $m=2,3$ is a sphere (compare Remark~\ref{rem:dimension}).

\begin{question} 
Are there manifolds of dimension $n=3,4$ as in Theorem~\ref{maintheorem}?
\end{question}

\addtocontents{toc}{\protect\setcounter{tocdepth}{1}}

\subsection*{Acknowledgements}
I would like to thank Igor Belegradek, Jean-Fran\c cois Lafont, and Loring Tu for helpful conversations.
I acknowledge support by the AMS and the Simons Foundation.

% % % % % % % % % % % % % % % % % % % % % % % % % % % % % % % % % % % % % % % % % % % % % % % % % % % % % % % % 
% % % % % % % % % % % % % % % % % % % % % % % % % % % % % % % % % % % % % % % % % % % % % % % % % % % % % % % % 
% % % % % % % % % % % % % % % % % % % % % % % % % % % % % % % % % % % % % % % % % % % % % % % % % % % % % % % % 

\section{Background on locally homogeneous geometric structures}\label{sec:gxstructures}
Let $G$ be a real Lie group and let $X$ be a homogeneous space for $G$. 
Let $M$ be a smooth manifold.
A \textit{locally homogeneous geometric structure modelled on $(G,X)$} (or in short a $(G,X)$--\textit{structure}) on $M$ can be defined in terms of local data, by a collection of local charts on $M$ with values in $X$, such that the change of coordinates is given by elements of $G$. 

Equivalently, a $(G,X)$--structure can be defined in terms of global data, by a representation $\rho:\pi_1(M)\to G$ (the \textit{holonomy representation}) and a $\rho$-equivariant local diffeomorphism $\dev:\widetilde M\to X$ (the \textit{developing map}), where $\widetilde M$ denotes the universal cover of $M$.
To see the connection between the local and the global point of view, one notes that precomposing the  developing map with local inverses to the universal covering projection $\widetilde{M} \to M$ provides an atlas of local charts.
We refer the reader to \cite{TH97,GO10} for details and additional references.

\begin{remark}
We are not imposing any additional conditions on a $(G,X)$-structure, such as being complete, uniformizable, etc. In particular, developing maps are not required to be covering maps on their image, and holonomies are not required to be faithful nor discrete.
\end{remark}

A main source of examples of $(G,X)$-structures comes from Riemannian metrics of constant sectional curvature.
These provide structures modelled on
the hyperbolic space, the Euclidean space, or the round sphere, depending on the sign of the curvature.
More generally, when the point stabilizers for the action of $G$ on $X$ are compact, one can produce a $G$-invariant Riemannian metric on $X$ and on manifolds with a $(G,X)$-structure.
But if points stabilizers are not compact, then the geometric structure is not induced by a Riemannian metric. These are the types of geometric structures we consider in this paper.

Let $\rp n$ be the $n$-dimensional real projective space, and let $\pglr$ denote its group of real projective transformations.
A  \textit{real projective structure}  is a $(\pglr,\rp n)$--structure. This is also known as a \textit{flat projective structure} in the literature.

Let $\confsphere ^n$ be the $n$-dimensional sphere, equipped with the standard conformal structure, and let $\conf$ be its group of conformal transformations (also known as M\"obius transformations). 
A \textit{flat conformal structure}   is a $(\conf,\confsphere ^n)$--structure.
Thanks to Liouville's theorem, in dimension at least $3$, a flat conformal structure can equivalently be defined by a conformal class of Riemannian metrics that are locally conformal to the flat Euclidean metric.

The classical geometries of constant curvature can be realized both inside real projective geometry and inside flat conformal geometry, so they are examples of both.
For instance, hyperbolic geometry can be realized inside $\rp n$ via the Klein model, and inside $\confsphere^n$ via the Poincar\'e disk model.

\subsection{An obstruction}\label{sec:obstruction}

While constructing examples of manifolds with real projective or flat conformal structures is a classical and popular problem, not much is known about how to construct manifolds that do not (see the introduction for an account of the work in this direction).
\c Coban in \cite{C19} has proposed the following obstruction, based on the observation that a developing map for a real projective structure on an $n$--manifold provides in particular a smooth immersion into $\rr^n$ for its simply connected submanifolds.

\begin{proposition}[Theorem 2.1 in \cite{C19}]\label{prop:projective obstruction}
Let $N$ be a smooth $n$--manifold containing a  simply connected  smooth $m$-submanifold $M$ (for some $m<n$) that does not admit any smooth immersion in $\rr^n$. Then $N$ does not admit any real projective structure.
\end{proposition}

The proof that \c Coban gave for this result also works in the case of flat conformal structures.
We include the proof for the sake of completeness.
A similar strategy in the study of flat conformal structures was also considered by Kapovich in \cite{KA04}.

\begin{proposition}\label{prop:conformal obstruction}
Let $N$ be a smooth $n$--manifold containing a  simply connected  smooth $m$-submanifold $M$ (for some $m<n$) that does not admit any smooth immersion in $\rr^n$. Then $N$ does not admit any flat conformal structure.
\end{proposition}
\begin{proof}
Assume by contradiction that $N$ admits a flat conformal structure. 
Let $\dev:\widetilde N\to \confsphere ^n$ be a developing map for it, and let $\pi: \widetilde N \to N$ be the universal covering projection.
Since $M$ is simply connected, $M$ lifts homeomorphically to $\widetilde N$.
Let $\widetilde M$ be such a lift of $M$ to $\widetilde N$.
The restriction of $\dev$ to $\widetilde M$ provides a smooth immersion of $\widetilde M$ into $\confsphere ^n$.
Since $M$ is a smooth submanifold of dimension $m<n$, $\dev (\widetilde M)$ cannot be the entire sphere, so it is  contained in the complement of a point $p\in \confsphere^n$.
So, the restriction of $\dev$ to $\widetilde M$  provides a smooth immersion of $\widetilde M$ into $\confsphere^n\setminus \{p\}=\rr^n$, which contradicts our assumption.
\end{proof}

\begin{remark}\label{rem:obstruction stable}
The obstruction from Propositions \ref{prop:projective obstruction} and \ref{prop:conformal obstruction} is stable, i.e. it passes to any cover of $N$.
\end{remark}

\begin{remark}\label{rem:obstruction all spheres}
As the reader might have noticed, no role is played here by the group $G$ or the holonomy representation.
Indeed, the same proof holds for any kind of $(G,X)$--structures in which $X$ is the sphere $S^n$, or a manifold covered by the sphere. 
For example, spherical CR geometry, i.e. $X=S^{2n-1}$ and $G=\operatorname{PU}(n,1)$. 
This is the geometry at infinity for complex hyperbolic space. 
Or spherical quaternionic geometry, i.e. $X=S^{4n-1}$ and $G=\operatorname{Sp}(n,1)$. 
This is the geometry at infinity for complex quaternionic space.
See \cite[\S II.10]{BH99} for details. 
(Note: here $S^n$  denotes the standard smooth sphere without any additional structure, while $\confsphere^n$ is reserved for the sphere $S^n$ equipped with its standard conformal structure.)
\end{remark}

% % % % % % % % % % % % % % % % % % % % % % % % % % % % % % % % % % % % % % % % % % % % % % % % % % % % % % % % 

% % % % % % % % % % % % % % % % % % % % % % % % % % % % % % % % % % % % % % % % % % % % % % % % % % % % % % % % 

\section{The manifolds}
\addtocontents{toc}{\protect\setcounter{tocdepth}{2}}
This section is devoted to the construction of the manifolds in Theorem~\ref{maintheorem}.
The construction is based on some hyperbolization procedures that we review in \S\ref{sec:strict hyperbolization} and \S\ref{sec:relative hyperbolization}. We prove Theorem~\ref{maintheorem} in \S\ref{sec:main}, and construct explicit examples in \S\ref{sec:examples}.

\subsection{Strict hyperbolization}\label{sec:strict hyperbolization}

Gromov introduced some hyperbolization procedures to turn a simplicial complex $K$ into a locally $\cat 0$ space (see \cite{G87,P91,DJ91}).
In \cite{CD95} Charney and Davis refined the construction to obtain a space that is locally $\cat{-1}$. Their construction is known as strict hyperbolization, because it provides a space of strictly negative curvature, as opposed to just non-positive curvature.
If $K$ is a simplicial complex,
we denote by $\hyperbolize K$ the space obtained via strict hyperbolization, and call it the \textit{hyperbolized complex}.

\begin{figure}[h]
\centering
\def\svgwidth{\columnwidth}
%% Creator: Inkscape 1.1.1 (3bf5ae0d25, 2021-09-20), www.inkscape.org
%% Accompanies image file 'hyperbolization_from_simplicial.pdf' (pdf, eps, ps)
\begingroup%
  \makeatletter%
  \providecommand\color[2][]{%
    \errmessage{(Inkscape) Color is used for the text in Inkscape, but the package 'color.sty' is not loaded}%
    \renewcommand\color[2][]{}%
  }%
  \providecommand\transparent[1]{%
    \errmessage{(Inkscape) Transparency is used (non-zero) for the text in Inkscape, but the package 'transparent.sty' is not loaded}%
    \renewcommand\transparent[1]{}%
  }%
  \providecommand\rotatebox[2]{#2}%
  \newcommand*\fsize{\dimexpr\f@size pt\relax}%
  \newcommand*\lineheight[1]{\fontsize{\fsize}{#1\fsize}\selectfont}%
  \ifx\svgwidth\undefined%
    \setlength{\unitlength}{2221.31980115bp}%
    \ifx\svgscale\undefined%
      \relax%
    \else%
      \setlength{\unitlength}{\unitlength * \real{\svgscale}}%
    \fi%
  \else%
    \setlength{\unitlength}{\svgwidth}%
  \fi%
  \global\let\svgwidth\undefined%
  \global\let\svgscale\undefined%
  \makeatother%
  \begin{picture}(1,0.28612687)%
    \lineheight{1}%
    \setlength\tabcolsep{0pt}%
    \put(0,0){\includegraphics[width=\unitlength,page=1]{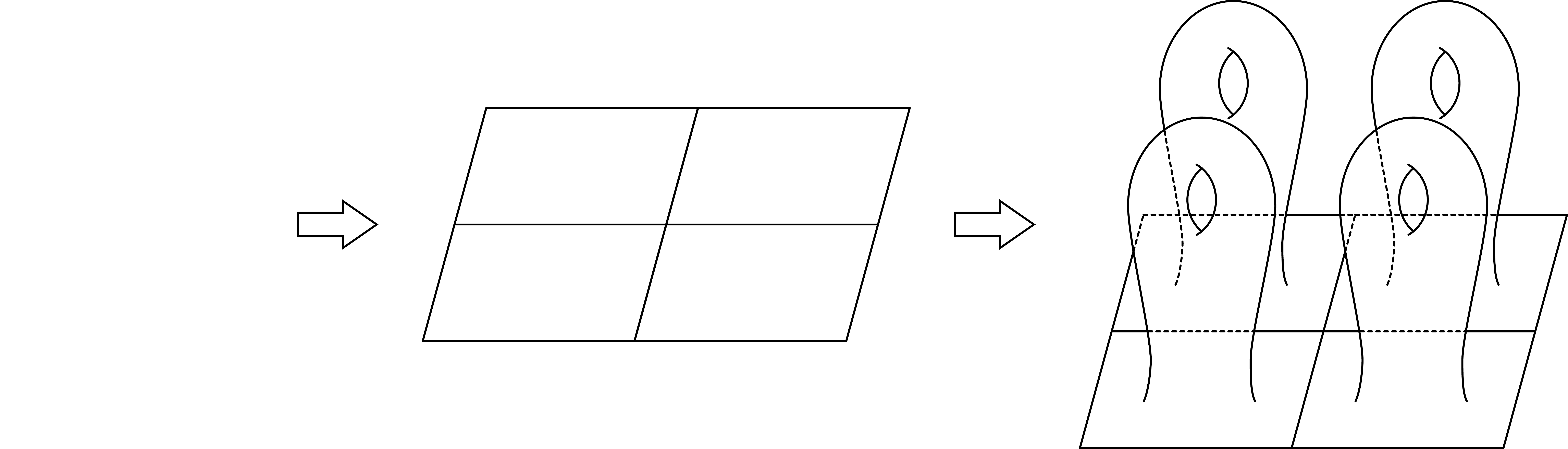}}%
    \put(0.70255649,0.00941637){\color[rgb]{0,0,0}\makebox(0,0)[lt]{\lineheight{1.25}\smash{\begin{tabular}[t]{l}$\hyperbolize K$\end{tabular}}}}%
    \put(0,0){\includegraphics[width=\unitlength,page=2]{hyperbolization_from_simplicial.pdf}}%
    \put(0.06128959,0.07417982){\color[rgb]{0,0,0}\makebox(0,0)[lt]{\lineheight{1.25}\smash{\begin{tabular}[t]{l}$K$\end{tabular}}}}%
  \end{picture}%
\endgroup%

    \caption{A local sketch of the strict hyperbolization procedure $K\mapsto \hyperbolize K$. First each simplex of $K$ is turned into a cubical complex, then every cube is replaced with a certain hyperbolic manifold with boundary and corners to obtain $\hyperbolize K$.}
    \label{fig:hyp}
\end{figure}

The main ingredient in \cite{CD95} is the existence, in each dimension $n$, of a certain compact hyperbolic $n$-manifold  with boundary and corners. 
This manifold is constructed via arithmetic methods.
Roughly speaking, $\hyperbolize K$ is obtained by replacing simplices of $K$ with copies of such hyperbolic manifold (see Figure~\ref{fig:hyp}, and \cite{CD95,BE07,LR22} for details).
This significantly changes the topology of $K$ (e.g. its fundamental group), but in a  controlled way. Indeed, the procedure $K\mapsto \hyperbolize K$ enjoys many properties, and we list the ones relevant for this paper.

\begin{enumerate}
    \item If $K$ is a smoothly triangulated smooth manifold, then so is $\hyperbolize K$.

    \item $\hyperbolize K$ is a locally $\cat{-1}$ and piecewise-hyperbolic cell complex. In particular, $\hyperbolize K$ is aspherical, and if $K$ is compact, then $\pi_1(\hyperbolize K)$ is a (non-elementary) Gromov hyperbolic group.

    \item \label{item:link} There is a continuous map $\cd:\hyperbolize K \to K$ that induces a PL-homeomorphism on the links, i.e. $\lk{p,\hyperbolize K} \cong \lk{\cd(p),K}$ for all $p\in \hyperbolize K$.

\end{enumerate}

% % % % % % % % % % % % % % % % % % % % % % % % % % % % % % % % % % % % % % % % % % % % % % % % % % % % % % % % 

\subsection{Relative strict hyperbolization}\label{sec:relative hyperbolization}

The strict hyperbolization procedure described in \S\ref{sec:strict hyperbolization} has a relative version that can be used to work relatively to a chosen subcomplex $L$ of $K$, i.e. without changing its topology.
This construction was introduced in \cite{G87,DJ91,DJW01,CD95}, and we find it convenient to follow the description in  \cite[\S 4]{BE07}.

\begin{figure}[h]
\centering
\def\svgwidth{.7\columnwidth}
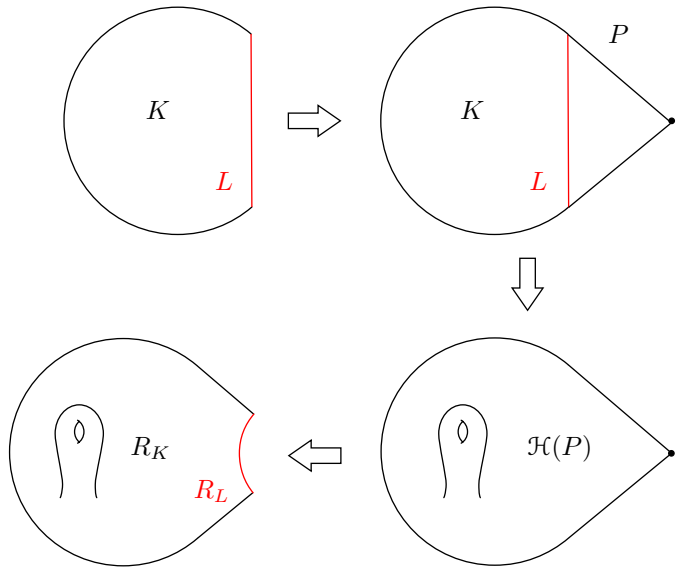
    \caption{The four steps involved in the relative strict hyperbolization procedure.}
    \label{fig:rel hyp}
\end{figure}

Let $K$ be a simplicial complex, and let $L$ be a (not necessarily connected) subcomplex of $K$.
Let $P$ be the simplicial complex obtained by coning-off the path-components of $L$, see Figure~\ref{fig:rel hyp}.
(If $L$ has more than one path-component, one can interpret this as ``attach a cone to each path-component individually'', or as ``attach a single cone point connecting to all the path-components simultaneously''. As explained in \cite[Remark 4.2]{BE07} this choice is irrelevant.)
Let $\hyperbolize P$ be the result of applying strict hyperbolization to $P$ (see \S\ref{sec:strict hyperbolization} and Figure~\ref{fig:hyp}).
Notice that $\hyperbolize P$ has some special point(s) corresponding to the cone point(s) of $P$.
Since hyperbolization preserves links (see \eqref{item:link} above), the link of such point(s) is isomorphic to the link of the corresponding point(s) of $P$, i.e. to the path-components of $L$.
Finally, let $R_K$ be the result of chopping off the cone point(s) of $\hyperbolize P$.
More formally, let $O$ be a small open neighborhood of the cone point(s), and let $R_K = \hyperbolize P \setminus O$.
We denote by $R_L$ the (topological) boundary of $R_K$ in $\hyperbolize L$.

The pair $(R_K,R_L)$ will be called the \textit{relative strict hyperbolization} of the pair $(K,L)$.
The procedure $(K,L)\mapsto (R_K,R_L)$ enjoys many properties, and we list the ones relevant for us.

\begin{enumerate}
    \item If $K$ is a smoothly triangulated smooth manifold with boundary $L$, then $R_K$ is a smooth manifold with boundary $R_L$, and $R_L$ is a subdivision of $L$.
    
    \item If $K$ is compact, then $\pi_1(R_K)$ is a non-elementary relatively hyperbolic group, relative to $\pi_1(L_i)$, where $L_i$ are the path-components of $R_L$ (see \cite[Theorem 1.1]{BE07}). 
    
    \item $R_K$ is aspherical if and only if each path-component of $R_L$ is aspherical.
    
\end{enumerate}

We note in particular that the natural metric on $\hyperbolize P$ is locally $\cat{-1}$, but the one on $R_K$ is not.
Moreover, it is known that $\pi_1(R_K)$ retracts to many free subgroups and does not have property (T) by \cite[Lemma 4.3]{BE07}.

% % % % % % % % % % % % % % % % % % % % % % % % % % % % % % % % % % % % % % % % % % % % % % % % % % % % % % % % 
\subsection{The main construction}\label{sec:main}
Throughout \S\ref{sec:main}, let $m\geq 4, n=m+1$ and let $M$ be a closed simply connected smooth $m$-manifold which does not admit an immersion into $\rr^n=\rr^{m+1}$.
Concrete examples are given by $M=\cp 2$ for $m=4$ and $M=\wu$ for $m=5$. For details and examples in higher dimension see \S~\ref{sec:examples}. 

\begin{proof}[Proof of Theorem~\ref{maintheorem}]

We provide a proof in the following three steps.

\subsubsection*{Construction of the manifold $N$}

Let $K$ be a smooth triangulation of $M\times [0,1]$.
Let $L=\partial K$. Note that each of the two path-components of $L$ is a copy of $M$.
Let $(R_K,R_L)$ be the result of applying relative strict hyperbolization to the pair $(K,L)$ (see \S\ref{sec:relative hyperbolization}).
Then $R_K$ is a compact  smooth manifold of dimension $n=m+1$ with boundary $R_L$.
Note that each of the two path-components of $R_L$ is a copy of $M$.
Let $N$ be the closed smooth $n$-manifold obtained by gluing the two boundary components of $R_K$ together.

\subsubsection*{Proof that $\pi_1(N)$ is hyperbolic}
By \cite[Theorem 1.1]{BE07} the fundamental group $\pi_1(R_K)$ is a non-elementary relatively hyperbolic group, which is hyperbolic relatively to the fundamental groups of the two boundary components.
These are copies of $M$, hence they are simply connected.
Therefore $\pi_1(R_K)$ is actually a non-elementary Gromov hyperbolic group.
Note that  $\pi_1(N) = \pi_1(R_K)\ast \zz$, hence $\pi_1(N)$ is a non-elementary Gromov hyperbolic group too.

\subsubsection*{Proof that $N$ does not admit real projective or flat conformal structures}
By construction $N$ is a smooth manifold of dimension $n=m+1$ and it contains a smooth $m$-submanifold $M$  which is simply connected and does not admit any immersion into $\rr^n=\rr^{m+1}$.
(This comes from the two boundary copies of $M$ that have been glued together.)
By Proposition~\ref{prop:projective obstruction},  $N$ does not admit any real projective structure.
By Proposition~\ref{prop:conformal obstruction},   $N$ does not admit any flat conformal structure.

\end{proof}

\begin{remark}\label{rem:dimension}
In the above construction, we used a closed simply connected $m$--manifold that does not immerse in $\rr^{m+1}$. 
This is not available in dimension $m= 3$, because such a manifold is a sphere by the solution to the Poincar\'e conjecture, and the $3$-sphere embeds in $\rr^4$. (Even before the solution to this conjecture, it was already known that any compact $3$-manifolds embeds in $\rr^4$ thanks to \cite[Theorem 6.7]{HI59}.)
This is why we started by assuming $m\geq 4$.
\end{remark}

\begin{remark}\label{rem:not strict but rel}
A manifold $N$ as in Theorem~\ref{maintheorem}  cannot be obtained  by strict hyperbolization (e.g. on the manifolds obtained in \cite{C19}). 
For instance, hyperbolized manifolds are aspherical, but $N$ is not aspherical, because $M$ is not aspherical.
To find an aspherical example, one could try to run the above construction starting with an aspherical $M$.
However, note that the lift $\widetilde M$ of $M$ to the universal cover $\widetilde N$ of $N$ would be a contractible submanifold.
Since a contractible manifold is in particular parallelizable,  by \cite[Theorem 6.3]{HI59} $\widetilde M$ actually admits an immersion into a Euclidean space of one dimension higher, so there is no way to use the obstruction from \S\ref{sec:obstruction}.
The point of using relative strict hyperbolization relative to a simply connected $M$ is exactly to preserve its topology. 
\end{remark}

\begin{remark}
The reader only interested in obstructing the existence of flat conformal structures should notice that manifolds with flat conformal structures have zero Pontryagin classes (see \cite{AV70}). 
In particular, Ontaneda's Riemannian hyperbolization provides examples of aspherical manifolds without flat conformal structures in any dimension $n\geq 4$ (see \cite{O20}.)
\end{remark}

% % % % % % % % % % % % % % % % % % 
% % % % % % % % % % % % % % % % % % 

\subsection{Examples}\label{sec:examples}

Let $M$ be a closed simply connected smooth $m$-manifold.
We are interested in the case in which $M$ does not admit a smooth immersion in $\rr^{m+1}$. By Remark~\ref{rem:dimension}, we assume $m\geq 4$.
Following \cite[\S4]{MS74}, we can use Stiefel-Whitney classes to obstruct the existence of smooth immersions in a Euclidean space.
Let us denote by $w_i(M)\in H^i(M,\zz_2)$ the $i$-th Stiefel-Whitney class of  $M$, and by $\overline w_i(M)$ the $i$-th  dual Stiefel-Whitney class.
The following statement is easily deduced from \cite[\S 4]{MS74}. We include a proof for the convenience of the reader.

\begin{lemma}\label{lem:swclass}
Let $M$ be a simply connected smooth $m$-manifold.
If $M$ admits a smooth immersion into $\rr^{m+1}$ then $w_2(M)=0$.
\end{lemma}
\begin{proof}
Suppose that $M$ admits a smooth immersion into $\rr^{m+1}$.
The normal bundle for the immersion has rank $1$, so all of its classes vanish above degree $1$.
Whitney duality theorem \cite[Lemma 4.2]{MS74} implies that the dual classes of $M$ vanish in the same range. 
In particular, we have that  $\overline{w}_2(M)=0$.
As in the proof of \cite[Lemma 4.1]{MS74}, we have $\overline{w}_2(M)=w_1^2(M)+w_2(M)$.
But  we also have $w_1(M)=0$, because $M$ is simply connected.
Therefore $w_2(M)=\overline{w}_2(M)=0$.
\end{proof}

We propose some well-known examples of closed simply connected smooth manifold with non-zero second Stiefel-Whitney class. 

\begin{example}[$4$-dimensional example]\label{ex:cp2}
Let $M=\cp 2$ be the $2$-dimensional complex projective space. 
By \cite[Corollary 11.15]{MS74} we have that $w_2(\cp 2) = \binom{3}{2} a = 3a = a$ where $a\in H^2(\cp 2,\zz_2)=\zz_2 $ is a generator. In particular, $w_2(\cp 2)\neq 0$.
\end{example}

\begin{example}[$5$-dimensional example]\label{ex:wu}
The homogeneous space $M=\wu$ is a closed simply connected smooth manifold of dimension $5$.
It is known that $w_2(M)\neq 0$ (see for instance \cite{BH58}, or \cite[Lemma 1.1(v)]{BA65}, where $M$ is identified with the Wu manifold and denoted  $X_{-1}$).
\end{example}

A way to get examples of higher dimension is to consider suitable homogeneous spaces (see \cite{BH58}), or to just take  products with spheres.
In particular, combining Examples~\ref{ex:cp2} and \ref{ex:wu} with the following lemma one can obtain manifolds of any dimension  $m\geq 6$ with the desired properties.

\begin{lemma}\label{lem:swclass_product_sphere}
Let $d\geq 2$.
If $M$ is a closed simply connected smooth manifold of dimension $k\geq 4$ with $w_2(M)\neq 0$, 
then $M\times S^d$ is a closed simply connected smooth manifold of dimension $m=k+d$ that does not admit a smooth immersion in $\rr^{m+1}$.
\end{lemma}
\begin{proof}
Clearly, $M\times S^d$ is a closed simply connected smooth manifold because both $M$ and $S^d$ are. 
It follows from Whitney product formula (see \cite[\S, p. 37-38]{MS74}) that
$$w_2(M\times S^d)=w_2(M) + w_1(M)w_1(S^d) + w_2(S^d).$$
Since for $d\geq 2$ we have that $w_1(S^d)=w_2(S^d)=0$, it follows that $w_2(M\times S^d)=w_2(M)\neq 0$.
By Lemma~\ref{lem:swclass} we can conclude that $M\times S^d$ does not admit any smooth immersion in $\rr^{k+d+1}=\rr^{m+1}$.
\end{proof}

%%%%%%%%%%%%%%%%%%%%%%%%%%%%%%%%%%%%%%%%%%
%%%%%%%%%%%%%%%%%%%%%%%%%%%%%%%%%%%%%%%%%%
%%%%%%%%%%%%%%%%%%%%%%%%%%%%%%%%%%%%%%%%%%

\printbibliography
 
\end{document}